\documentclass[10pt]{amsart}

\usepackage{amssymb, latexsym, amsfonts, pigpen, enumitem, mathrsfs, amsthm, bbm, hyperref}
\usepackage[matrix,arrow,curve]{xy}

\newtheorem{theorem}{Theorem}[section]
\newtheorem{lemma}[theorem]{Lemma}
\newtheorem{corollary}[theorem]{Corollary}

\theoremstyle{definition}
\newtheorem{definition}[theorem]{Definition}

\theoremstyle{remark}
\newtheorem{remark}[theorem]{Remark}

\numberwithin{equation}{section}

\newcommand{\Sp}{{\mathrm{Sp}}}
\newcommand{\SL}{{\mathrm{SL}}}
\newcommand{\SLc}{{\mathrm{SL^c}}}

\newcommand{\Z}{{\mathbb{Z}}}

\newcommand{\K}{\mathrm{K}}

\newcommand{\A}{\mathbb{A}}

\newcommand{\PP}{\mathbb{P}}

\newcommand{\rank}{\operatorname{rank}}
\newcommand{\Th}{{Th}}
\newcommand{\thc}{{th}}
\newcommand{\id}{\operatorname{id}}

\newcommand{\triv}{\mathbf{1}}

\newcommand{\struct}{\mathcal{O}}
\newcommand{\Gm}{{\mathbb{G}_m}}
\newcommand{\GL}{\mathrm{GL}}
\newcommand{\SH}{\mathcal{SH}}

\newcommand{\T}{\mathrm{T}}

\newcommand{\Spec}{\operatorname{Spec}}

\newcommand{\Sph}{\mathrm{S}}

\newcommand{\Sm}{\mathcal{S}m}

\newcommand{\Smk}{\mathcal{S}m_S}
\newcommand{\SHk}{\mathcal{SH}(S)}
\newcommand{\HProj}{\mathrm{HP}}

\newcommand{\mcE}{\mathcal{E}}

\newcommand{\phs}{\mathcal{P}}
\newcommand{\hyperb}{\mathcal{H}}

\newcommand{\rightiso}{\xrightarrow{\simeq}}

\newcommand{\pour}{\ar@{}[ur]|(0.2){\text{\pigpenfont G}}}
\newcommand{\podr}{\ar@{}[dr]|(0.2){\text{\pigpenfont A}}}

\begin{document}

\title{SL-oriented cohomology theories}


%    author one information
\author{Alexey Ananyevskiy}
\address{St. Petersburg Department, Steklov Math. Institute, Fontanka 27, St. Petersburg 191023 Russia, and Chebyshev Laboratory, St. Petersburg State University, 14th Line V.O., 29B, St. Petersburg 199178 Russia}
\email{alseang@gmail.com}

\thanks{The research is supported by Young Russian Mathematics award, by "Native towns", a social investment program of PJSC "Gazprom Neft", and by RFBR grant 18-31-20044. A part of the work was done during the author's stay at the University of Oslo, the visit was supported by the RCN Frontier Research Group Project no. 250399 "Motivic Hopf equations".}

%\subjclass[2000]{Primary }
%    The 2010 edition of the Mathematics Subject Classification is
%    now available.  If you are citing a classification from the
%    new scheme, use the following input coding instead.
\subjclass[2010]{Primary 14F42; Secondary 55P99, 55R40}

\date{}

\begin{abstract}
	We show that a representable motivic cohomology theory admits a unique normalized $\SLc$-orientation if the zeroth cohomology presheaf is a Zariski sheaf. We also construct Thom isomorphisms in $\SL$-oriented cohomology for $\SLc$-bundles and obtain new results on the $\eta$-torsion characteristic classes, in particular, we prove that the Euler class of an oriented bundle admitting a (possibly non-orientable) odd rank subbundle is annihilated by the Hopf element.
\end{abstract}

\maketitle

\section{Introduction}

The setting of oriented cohomology theories in algebraic geometry \cite{PS03,S07a,LM07} is rather well developed, and currently the arising characteristic classes, pushforward maps and operations between oriented cohomology theories, governed by variants of Riemann--Roch theorem \cite{PS04,S07b}, are understood in many details. Particular examples of oriented cohomology theories, such as Chow groups, motivic cohomology and Quillen $\K$-theory proved to be extremely useful tools for studying various questions of algebro-geometric nature. In the recent years attention was drawn to the quadratic refinements of these classical cohomology theories, namely to Chow--Witt groups, Milnor--Witt motivic cohomology and hermitian $\K$-theory. The latter theories are genuinely non-orientable, the projective bundle formula fails for them. Nevertheless these theories share a certain generalized orientation property that was introduced by I.~Panin and Ch.~Walter \cite{PW10a,PW10b,PW10c}. 

One of the existing approaches to oriented cohomology theories is based on the existence of Thom isomorphisms for vector bundles \cite{PS03,S07a}. Roughly speaking, a cohomology theory $A^{*,*}(-)$ is oriented if for every vector bundle $E$ over $X$ one has a so-called Thom isomorphism $A^{*,*}(X)\cong A^{*+2n,*+n}_X(E)$, where $A^{*,*}(X)$ stands for the cohomology of $X$, $A^{*+2n,*+n}_X(E)$ is the cohomology of $E$ supported on the zero section and $n=\rank E$. Additionally one asks for some natural properties of these Thom isomorphisms -- functoriality and compatibility with direct sums of vector bundles. The notion of orientation may be generalized as it is done in \cite{PW10a,PW10b} by postulating the existence of Thom isomorphisms only for vector bundles with a certain additional structure -- symplectic bundles (symplectic orientation), vector bundles with trivialized determinant line bundle ($\SL$-orientation), vector bundles with the chosen square root of the determinant line bundle ($\SLc$-orientation). One may also introduce the notion of $G$-orientation for a general family of group sheaves $G=\{G_n\}_{n\in\mathbb{N}}$ following Remark~\ref{rem:reducedgroup} and Definition~\ref{def:orientation} of the current paper, but at the moment we are mostly interested in $\SLc$ and $\SL$-orientations. 

It is easy to see that an $\SLc$-oriented cohomology theory is naturally $\SL$-oriented since if the determinant of a vector bundle is trivial then it clearly has a square root. But it is rather surprising that an $\SL$-oriented cohomology theory admits Thom isomorphisms for vector $\SLc$-bundles as well:

\begin{theorem}[Theorem~\ref{thm:SLvsSLc}]
	Let $A\in\SHk$ be an $SL$-oriented spectrum and $\mcE$ be a rank $n$ vector $SL^c$-bundle over $X\in\Smk$. Then there exists an isomorphism 
	\[
	A^{*,*}(X)\rightiso A^{*+2n,*+n}_X(\mcE).
	\]
\end{theorem}
\noindent The reason for this is that for a line bundle $L$ the Thom spaces $\Th(L)$ and $\Th(L^\vee)$ are naturally isomorphic which allows one to cancel a twist by the square of a line bundle in $\SL$-oriented cohomology. The above theorem yields that $\SLc$ and $\SL$-orientations are closely related and because of that we do not loose much information focusing on $\SL$-orientation.

In order to show that a cohomology theory is $\SL$-oriented one usually has to come up with a construction of Thom classes for vector $\SL$-bundles which in particular cases may be quite tricky. In the current paper we show that one always has a unique $\SL$-orientation (and $\SLc$-orientation) for a commutative ring spectrum $A$ provided that $A^{0,0}(-)$ is a sheaf. Namely, we prove the following theorem.

\begin{theorem}[{Theorem~\ref{thm:sheafSL} and Corollary~\ref{cor:sheafSL}}]
	Let $A\in \SHk$ be a commutative ring spectrum and suppose that $A^{0,0}(-)$ is a sheaf in the Zariski topology. Then $A$ admits unique normalized $SL^c$ and $SL$-orientations.
\end{theorem}
\noindent Note that if the base scheme is the spectrum of a field then $A^{0,0}(-)$ is a sheaf in particular for the spectra representing sheaf cohomology with the coefficients in a homotopy module, Voevodsky's motivic cohomology and Milnor-Witt motivic cohomology. The proof of the above theorem roughly goes as follows. First we define Thom classes for trivial vector $\SLc$-bundles choosing an arbitrary trivialization and pulling back the appropriate suspension $\Sigma^n_\T 1$. Then we show that the class defined in such a way does not depend on the choice of a trivialization because the action of a matrix from $\mathrm{SL^c_n}$ (a matrix with the determinant being a perfect square) on the Thom space of a trivial bundle is locally homotopy equivalent to the trivial action. This allows one to patch the Thom class of a nontrivial vector $\SLc$-bundle from a trivializing cover.

We also improve some results on characteristic classes for $\SL$-oriented cohomology theories obtained in \cite{An15} where it was mostly assumed that the Hopf element $\eta$ is inverted in the coefficients. Extending the proof of \cite[Proposition~7.2]{Lev17} to the general setting of $\SL$-oriented cohomology theories we obtain the following theorem.
\begin{theorem}[{Theorem~\ref{thm:Euler_torsion}}]
	Let $A\in \SHk$ be an $SL$-oriented spectrum and $\mcE=(E,\lambda)$ be a vector $SL$-bundle over $X\in \Smk$. Suppose that there exists an isomorphism of vector bundles $E\cong E_1\oplus E_2$ with $E_1$ being of odd rank. Then $\eta \cup e(\mcE)=0$.
\end{theorem}
\noindent As a corollary we obtain the following properties of the Borel and Pontryagin classes.
\begin{corollary}[{Corollary~\ref{cor:Pontryagin}}]
	Suppose that $S=\Spec k$ for a field $k$ of characteristic different from $2$. Let $A\in \SH(k)$ be an $SL$-oriented spectrum. Then the following holds after inverting $\eta$ in the coefficients.
	\begin{enumerate}
		\item 
		$b_{2i+1}(\hyperb(E))=0, i\ge 0,$ for the hyperbolic symplectic bundle $\hyperb(E)$ associated to a vector bundle $E$ over $X\in\Sm_k$.
		\item 
		$p_t(E\oplus E')=p_t(E)p_t(E')$ for vector bundles $E,E'$ over $X\in\Sm_k$.
		\item 
		$p_i(E)=0$ for a rank $n$ vector bundle $E$ over $X\in\Sm_k$ and $i>\tfrac{n}{2}$.
	\end{enumerate}
	Here $b_{2i+1}$, $p_i$ and $p_t$ denote the Borel, Pontryagin and total Pontryagin classes of the respective bundles, see Definitons~\ref{def:Borel} and~\ref{def:Pontryagin} or \cite[Definition~14.1]{PW10a}, \cite[Definition~7]{An15} and \cite[Definition~5.6]{HW17}.
\end{corollary}
\noindent This generalizes \cite[Corollary~3]{An15} where the same properties were obtained assuming the determinant of $E$ to be trivial.

The paper is organized as follows. In Section~\ref{section:Gbundles} we recall the notions of symplectic, $\SL$ and $\SLc$-bundles and outline the general setting of vector bundles with a chosen reduction of the structure group. In the next section we give a uniform definition of $G$-oriented cohomology theories and prove some basic properties. In Section~\ref{section:SLcSL} we show that an $\SL$-oriented cohomology theory admits Thom isomorphisms for vector $\SLc$-bundles. In Section~\ref{section:orientsheaf} we construct a normalized $\SLc$ and $\SL$-orientations on a commutative ring spectrum provided that $A^{0,0}(-)$ is a sheaf and show that such orientations are unique. In the next section we recall some well-known facts on the Hopf map and related elements. Then in the last section we show that the Euler class of a vector $\SL$-bundle admitting an odd rank vector subbundle is annihilated by $\eta$ and derive consequences on the Pontryagin classes.

Throughout the paper we employ the following assumptions and notations.

\begin{tabular}{l|l}
	$S$ & a quasi-compact separated scheme \\
	$\Smk$ & the category of smooth schemes over $S$\\
	$\SHk$ & the motivic stable homotopy category over $S$ \cite{MV99,V98}\\
	$A^{*,*}(-)$ & the cohomology theory represented by $A\in\SHk$, \\ 
	& with the multiplicative structure given by \cite[Section~3]{An17}\\
	$\Th(E)$ & the Thom space of a vector bundle $E$\\
	$\triv_X$ & the trivialized line bundle over $X\in\Smk$
\end{tabular}

\section{Vector bundles with additional structure} \label{section:Gbundles}
\begin{definition} \label{def:hyperb}
	A \textit{rank $2n$ symplectic bundle} (or \textit{vector $Sp$-bundle}) over $X\in\Smk$ is a pair $\mathcal{E}=(E,\phi)$ with $E$ being a rank $2n$ vector bundle over $X$ and $\phi$ being a symplectic form on $E$. An \textit{isomorphism of symplectic bundles} $\theta\colon (E,\phi)\rightiso (E',\phi')$ is an isomorphism of vector bundles $\theta\colon E\rightiso E'$ such that the diagram
	\[
	\xymatrix{
		E \ar[r]^\theta \ar[d]_\simeq & E' \ar[d]^\simeq \\
		E^\vee & (E')^\vee \ar[l]_{\theta^\vee}
	}
	\]
	commutes, where the vertical isomorphisms are given by the symplectic forms. The \textit{sum of symplectic bundles} is induced by the direct sum of vector bundles together with the orthogonal sum of symplectic forms,
	\[
	(E,\phi)\oplus (E',\phi') = (E\oplus E',\phi\perp\phi').
	\]
	
	For a vector bundle $E$ over $X\in\Smk$ the \textit{associated hyperbolic symplectic bundle} (or \textit{hyperbolization of $E$}) is the symplectic bundle
	\[
	\hyperb(E)= \left(E\oplus E^{\vee}, \begin{pmatrix}
	0 & 1 \\ -1 & 0
	\end{pmatrix}\right).
	\]	
	The \textit{trivialized} rank $2n$ symplectic bundle over $X\in\Smk$ is 
	\[
	\triv^{\Sp,2n}_X=\hyperb(\triv_X)^{\oplus n} =\left(\triv_{X}^{\oplus 2n}, \begin{pmatrix}
	0 & 1 \\ -1 & 0
	\end{pmatrix}^{\oplus n}\right).
	\]
	A \textit{trivial} rank $2n$ symplectic bundle over $X\in\Smk$ is a rank $2n$ symplectic bundle over $X\in\Smk$ isomorphic to the trivialized rank $2n$ symplectic bundle $\triv^{\Sp,2n}_X$.

\end{definition}

\begin{definition} \label{def:SLbundle}
	A \textit{rank $n$ special linear vector bundle} (or \textit{vector $SL$-bundle}) over $X\in\Smk$ is a pair $\mathcal{E}=(E,\lambda)$ with $E$ being a rank $n$ vector bundle over $X$ and $\lambda\colon \det E \rightiso \triv_X$ being an isomorphism of line bundles. An \textit{isomorphism of vector $SL$-bundles} $\theta\colon (E,\lambda)\rightiso (E',\lambda')$ is an isomorphism of vector bundles $\theta\colon E\rightiso E'$ such that $\lambda= \lambda'\circ \det \theta $,
	\[
	\xymatrix{
		\det E \ar[rr]^{\det \theta}_{\simeq} \ar[dr]^\simeq_\lambda & & \det E' \ar[dl]_\simeq^{\lambda'} \\
		& \triv_X &
	}
	\]
	The \textit{sum of special linear vector bundles} is induced by the direct sum of vector bundles together with the canonical isomorphisms $\det (E\oplus E')\cong \det E \otimes \det E'$ and $\triv_X\otimes \triv_X\cong\triv_X$ which we omit from the notation,
	\[
	(E,\lambda)\oplus (E',\lambda') = (E\oplus E',\lambda\otimes \lambda').
	\]
	The \textit{trivialized} rank $n$ vector $\SL$-bundle over $X\in\Smk$ is $\triv^{\SL,n}_X=(\triv_X,\id)^{\oplus n}$. A \textit{trivial} rank $n$ vector $\SL$-bundle over $X\in\Smk$ is a rank $n$ vector $\SL$-bundle over $X\in\Smk$ isomorphic to the trivialized rank $n$ vector $\SL$-bundle. 
\end{definition}

\begin{definition} \label{def:SLc}
	A \textit{rank $n$ vector $SL^c$-bundle} over $X\in\Smk$ is a triple $\mathcal{E}=(E,L,\lambda)$ with $E$ being a rank $n$ vector bundle over $X$, $L$ being a line bundle over $X$ and $\lambda\colon \det E \rightiso L^{\otimes 2}$ being an isomorphism of line bundles. An \textit{isomorphism of vector $SL^c$-bundles} $\theta\colon (E,L,\lambda)\rightiso (E',L',\lambda')$ is a pair $\theta=(\underline{\theta}\colon E\rightiso E',\,\hat{\theta}\colon L\rightiso L')$ consisting of an isomorphism of vector bundles and an isomorphism of line bundles such that the following diagram commutes.
	\[
	\xymatrix{
		\det E \ar[r]^{\det \underline{\theta}}_{\simeq} \ar[d]^\simeq_\lambda & \det E' \ar[d]_\simeq^{\lambda'} \\
		L^{\otimes 2} \ar[r]_{\hat{\theta}^{\otimes 2}}^{\simeq} & (L')^{\otimes 2}
	}
	\]
	We will usually abuse the notation omitting the underline and denoting the isomorphism $\underline{\theta}\colon E\rightiso E'$ by $\theta$ itself. 
	
	The \textit{sum of vector $SL^c$-bundles} is induced by the direct sum of vector bundles together with the canonical isomorphisms 
	\[
	\det (E\oplus E')\simeq \det E \otimes \det E', \quad (L\otimes L')^{\otimes 2}\simeq L^{\otimes 2} \otimes (L')^{\otimes 2}
	\]
	which we omit from the notation,
	\[
	(E,L,\lambda)\oplus (E',L',\lambda') = (E\oplus E',L\otimes L', \lambda\otimes \lambda').
	\]
	The \textit{trivialized} rank $n$ vector $\SLc$-bundle over $X\in\Smk$ is $\triv^{\SLc,n}_X=(\triv_X,\triv_X,\theta)^{\oplus n}$ with $\theta\colon \triv_X\rightiso \triv_X^{\otimes 2}$ being the canonical isomorphism. A \textit{trivial} rank $n$ vector $\SLc$-bundle over $X\in\Smk$ is a rank $n$ vector $\SLc$-bundle over $X\in\Smk$ isomorphic to the trivialized rank $n$ vector $\SLc$-bundle. 
\end{definition}

\begin{definition}
	Following the above notation we sometimes refer to rank $n$ vector bundles over $X\in\Smk$ as \textit{rank $n$ vector $GL$-bundles}. The \textit{trivialized} rank $n$ vector $\GL$-bundle over $X\in\Smk$ is $\triv^{\GL,n}_X=\triv_X^{\oplus n}$. A \textit{trivial} rank $n$ vector $\GL$-bundle over $X\in\Smk$ is a vector bundle isomorphic to $\triv_X^{\oplus n}$.
\end{definition}

\begin{definition} \label{def:Pfaffian}
	The hyperbolization construction from Definition~\ref{def:hyperb} gives rise to a morphism of groupoids
	\[
	\{\text{vector bundles over X} \} \xrightarrow{\hyperb}   \{\text{symplectic bundles over X} \}.
	\]
	Let $(E,\phi)$ be a symplectic bundle over $X\in \Smk$. The Pfaffian of $\phi$ induces an isomorphism of line bundles $\lambda_{\phi}\colon \det E \rightiso \triv_X$ (see, for example, the discussion above \cite[Definition~4.5]{An16a}) giving rise to the \textit{associated $SL$-bundle} $(E,\lambda_{\phi})$. This rule can be promoted in a canonical way to a morphism of groupoids
	\[
	\{\text{symplectic bundles over X} \} \xrightarrow{\Phi}   \{\text{vector $\SL$-bundles over X} \}.
	\]
	Let $(E,\lambda)$ be a vector $\SL$-bundle over $X\in \Smk$. The canonical isomorphism $\theta\colon \triv_X\rightiso \triv_X^{\otimes 2}$ gives rise to the \textit{associated $SL^c$-bundle} $(E,\triv_X,\theta \circ \lambda)$. This rule can be promoted in a canonical way to a morphism of groupoids
	\[
	\{\text{vector $\SL$-bundles over X} \} \xrightarrow{\Psi}   \{\text{vector $\SLc$-bundles over X} \}.
	\]	
	Let $(E,L,\lambda)$ be a vector $\SLc$-bundle over $X\in \Smk$. Forgetting about the additional structure $(L,\lambda)$ we obtain a morphism of groupoids
	\[
	\{\text{vector $\SLc$-bundles over X} \} \xrightarrow{\Theta}   \{\text{vector bundles over X} \}.
	\]		
\end{definition}

\begin{lemma} \label{lem:pfaf}
	In the notation of Definition~\ref{def:Pfaffian} there are canonical isomorphisms
	\begin{enumerate}
		\item
		$\hyperb(E_1\oplus E_2)\cong \hyperb(E_1) \oplus \hyperb(E_2)$ for vector bundles $E_1,E_2$,		
		\item
		$\Phi(\mathcal{E}_1\oplus \mathcal{E}_2)\cong \Phi(\mathcal{E}_1) \oplus \Phi(\mathcal{E}_2)$ for symplectic bundles $\mathcal{E}_1,\mathcal{E}_2$,
		\item
		$\Psi(\mathcal{E}_1\oplus \mathcal{E}_2)\cong \Psi(\mathcal{E}_1) \oplus \Psi(\mathcal{E}_2)$ for vector $SL$-bundles $\mathcal{E}_1,\mathcal{E}_2$,
		\item
		$\Theta(\mathcal{E}_1\oplus \mathcal{E}_2)\cong \Theta(\mathcal{E}_1) \oplus \Theta(\mathcal{E}_2)$ for vector $SL^c$-bundles $\mathcal{E}_1,\mathcal{E}_2$,		
		\item 
		$\Phi(\triv^{\Sp,2n}_X)\cong \triv^{\SL,2n}_X$,
		\item
		$\Psi(\triv^{\SL,n}_X)\cong \triv^{\SLc,n}_X$,
		\item
		$\Theta(\triv^{\SLc,n}_X)\cong \triv_X^{\oplus n}$.
	\end{enumerate}
\end{lemma}
\begin{proof}
	Straightforward.
\end{proof}

\begin{lemma} \label{lem:locallytrivial}
	Let $G=Sp,\,SL$ or $SL^c$. For a rank $n$ vector $G$-bundle $\mcE$ over $X\in \Smk$ and a point $x\in X$ there exists a Zariski open subset $U\subseteq X$ such that $x\in U$ and the restriction $\mcE|_{U}$ is a trivial rank $n$ vector $G$-bundle.
\end{lemma}
\begin{proof}
	It is sufficient to show that over a local ring every vector $G$-bundle is trivial. Let $W=\Spec R$ be the spectrum of a local ring. Recall that every vector bundle (without additional structure) over $W$ is trivial.
	
	$G=Sp$: follows from \cite[Chapter~1, Corollary~3.5]{MH73}.
	
	$G=SL$: we need to show that for a vector $\SL$-bundle $(\triv_W^{\oplus n}, \lambda)$ there exists an isomorphism of vector $\SL$-bundles $f\colon (\triv_W^{\oplus n}, \lambda)\rightiso \triv^{\SL,n}_W$. Put $\triv^{\SL,n}_W=(\triv_W^{\oplus n},\lambda_{\mathrm{triv}})$. One may take $f\colon \triv_W^{\oplus n}\rightiso \triv_W^{\oplus n}$ to be the isomorphism given by the diagonal matrix $\mathrm{diag}(\lambda_{\mathrm{triv}}\circ \lambda^{-1},1,1,\hdots,1)$.
	
	$G=SL^c$: we need to show that for a vector $\SLc$-bundle $(\triv_W^{\oplus n},\triv_W, \lambda)$ there exists an isomorphism of vector $\SLc$-bundles $f\colon (\triv_W^{\oplus n}, \triv_W,\lambda)\rightiso \triv^{\SLc,n}_W$. Put $\triv^{\SLc,n}_W=(\triv_W^{\oplus n}, \triv_W,\lambda_{\mathrm{triv}})$. One may take $f=(\underline{f},\id)$ with $\underline{f}\colon \triv_W^{\oplus n}\rightiso \triv_W^{\oplus n}$ given by the diagonal matrix $\mathrm{diag}(\theta^{-1}\circ\lambda_{\mathrm{triv}}\circ \lambda^{-1}\circ \theta,1,\hdots,1)$ with $\theta\colon\triv_X\rightiso \triv_X^{\otimes 2}$ being the canonical isomorphism.
\end{proof}

\begin{remark} \label{rem:torsors}
	Let $\mathrm{SL}^{\mathrm{c}}_{n}$ be the kernel of the homomorphism 
	\[
	\GL_n\times \Gm \to \Gm,\quad (g,t)\mapsto t^{-2}\det g,
	\]
	(cf. \cite[Definition~3.3]{PW10b}). It is easy to see that there are canonical equivalences of groupoids
	\begin{gather*}
	\{\text{rank $2n$ symplectic bundles over X} \} \stackrel{\simeq}{\longleftrightarrow} \{\text{$\Sp_{2n}$-torsors over X} \},\\
	\{\text{rank $n$ vector $\SL$ bundles over X} \} \stackrel{\simeq}{\longleftrightarrow} \{\text{$\SL_{n}$-torsors over X} \},\\
	\{\text{rank $n$ vector $\SLc$-bundles over X} \} \stackrel{\simeq}{\longleftrightarrow} \{\text{$\mathrm{SL}^{\mathrm{c}}_{n}$-torsors over X} \},	
	\end{gather*}
	with the torsors considered in the Zariski topology. Under the above equivalences the sum of vector $G$-bundles corresponds to the canonical homomorphisms
	\[
	\Sp_{2n}\times \Sp_{2m}\to \Sp_{2(n+m)},\quad \SL_{n}\times \SL_{m}\to \SL_{n+m},\quad\mathrm{SL}^{\mathrm{c}}_{n}\times \mathrm{SL}^{\mathrm{c}}_{m}\to \mathrm{SL}^{\mathrm{c}}_{n+m},
	\]
	while the morphisms $\hyperb$, $\Phi$, $\Psi$ and $\Theta$ from Definition~\ref{def:Pfaffian} correspond to the canonical homomorphisms
	\[
	\GL_n\to \Sp_{2n},\quad \Sp_{2n}\to \SL_{2n},\quad \SL_{n}\to \mathrm{SL}^{\mathrm{c}}_{n},\quad \mathrm{SL}^{\mathrm{c}}_{n}\to \GL_n.
	\]
\end{remark}

\begin{remark} \label{rem:reducedgroup}
	Let $G$ be a Zariski sheaf of groups on $\Smk$ and $\rho\colon G\to \GL_n$ be a homomorphism. Following Remark~\ref{rem:torsors} one can define the notion of a \textit{vector bundle with the structure group reduced to $G$} or a \textit{vector $G$-bundle} as a triple $(E,\phs,\psi)$ where 
	\begin{enumerate}
		\item 
		$E$ is a vector bundle over $X$, 
		\item
		$\phs$ is a $G$-torsor over $X$,
		\item
		$\psi\colon \triv_X^{\oplus n}\times_{G} \phs \xrightarrow{\simeq} E$ is an isomorphism of vector bundles.
	\end{enumerate}
	The case of $G=\mathrm{St}_n$ the Steinberg group seems to be of particular interest since, as the author learned from F.~Morel and A.~Sawant, $\mathrm{St}_n$ is closely related to the universal $\A^1$-covering group of $\SL_n$ whence the notion of a vector bundle with the structure group reduced to $\mathrm{St}_n$ may be related to the notion of Spin structure from classical topology. 
\end{remark}

\section{Orientations and twisted cohomology}
From now on we assume $G=\GL,\,\Sp,\,\SL$ or $\SLc$.

\begin{definition}
	Let $A\in \SHk$ be a spectrum and $E$ be a rank $n$ vector bundle over $X\in\Smk$. Denote $A^{*,*}(-;E)$ the presheaf on the slice category $\Smk/X$ (i.e. the category of morphisms of $S$-schemes $Y\xrightarrow{p} X$ with $Y\in\Smk$) given by
	\[
	[Y\xrightarrow{p} X]\mapsto A^{*+2n,*+n}(\Th(p^*E)).
	\]
	We refer to this presheaf as \textit{$A$-cohomology groups twisted by $E$}. For a rank $n$ vector $G$-bundle $\mcE$ we abuse the notation writing $A^{*,*}(-;\mcE)$ for $A^{*,*}(-;E)$ with $E$ being the vector bundle obtained from $\mcE$ forgetting the additional structure.
\end{definition}

\begin{remark}
	For $A\in\SHk$ and $X\in\Smk$ the suspension isomorphism
	\[
	\Sigma^n_\T\colon A^{*,*}(-)\rightiso A^{*+2n,*+n}((-)_+ \wedge \T^{\wedge n})
	\]
	of presheaves on $\Smk$ induces an isomorphism  
	\[
	\Sigma^n_\T\colon A^{*,*}(-)\rightiso A^{*,*}(-;\triv^{\oplus n}_X)
	\]
	of presheaves on the slice category $\Smk/X$.
\end{remark}

\begin{definition}[{cf. \cite[Definition~3.1.1]{PS03}, \cite[Definition~2.2.1]{S07a}, \cite[Definition~14.2]{PW10a}, \cite[Definitions~5.1 and~12.1]{PW10b} and \cite[Definition~4]{An15}}] \label{def:orientation}
	A \textit{normalized $G$-orientation} of a commutative ring spectrum $A\in\SHk$ is a rule which assigns to each rank $n$ vector $G$-bundle $\mathcal{E}$ over $X\in\Smk$ an element 
	\[
	\thc(\mathcal{E})\in A^{0,0}(X;\mcE)=A^{2n,n}(\Th(\mathcal{E}))
	\]
	with the following properties:
	\begin{enumerate}
		\item \label{def:orientation_iso}
		For an isomorphism $\theta\colon \mathcal{E}\xrightarrow{\simeq} \mathcal{E}'$ of vector $G$-bundles over $X\in\Smk$ one has 
		\[
		\thc(\mathcal{E})=\theta^A \thc(\mathcal{E}')
		\]
		where 
		$
		\theta^A\colon  A^{0,0}(X;\mcE')=A^{2n,n}(\Th(\mcE'))\to A^{2n,n}(\Th(\mcE))= A^{0,0}(X;\mcE)
		$
		is the pullback induced by $\theta$.
		\item \label{def:orientation_pull}
		For a morphism $f\colon Y\to X$ of smooth $S$-schemes and a vector $G$-bundle $\mcE$ over $X$ one has 
		\[
		f^A \thc(\mathcal{E})=\thc(f^*\mathcal{E})
		\]
		where 
		$
		f^A\colon  A^{0,0}(X;\mcE)=A^{2n,n}(\Th(\mcE))\to A^{2n,n}(\Th(f^*\mcE))= A^{0,0}(Y;f^*\mcE)
		$
		is the pullback induced by the morphism of total spaces $f^*\mathcal{E}\to \mcE$.
		\item \label{def:orientation_mult}
		For vector $G$-bundles $\mathcal{E}$ and $\mathcal{E}'$ over $X\in\Smk$ one has 
		\[
		\thc(\mathcal{E}\oplus \mathcal{E}')=q_1^A \thc(\mathcal{E})\cup q_2^A \thc(\mathcal{E}')
		\]
		where $q_1,q_2$ are the projections from $\mathcal{E}\oplus \mathcal{E}'$ to its factors.
		\item \label{def:orientation_norm}
		$G=\GL,\,\SL$ or $\SLc$:
		\[
		\thc(\triv^{G,1}_S)=\Sigma_{\T} 1 \in A^{0,0}(S;\triv_S).
		\]
		
		\noindent $G=\Sp$:
		\[
		\thc(\triv^{G,2}_S)=\Sigma^2_{\T} 1 \in A^{0,0}(S;\triv_S).
		\]
	\end{enumerate}
	We refer to the classes $\thc(\mathcal{E})$ as \textit{Thom classes}. A commutative ring spectrum $A$ with a chosen normalized $G$-orientation is called a \textit{$G$-oriented spectrum}. 
\end{definition}

\begin{lemma} \label{lem:orientationshomo}
	Let $A\in \SHk$ be a commutative ring spectrum. Then normalized orientations of $A$ for different structure groups induce each other as follows.
	\[
	\text{$GL$-oriented}\Rightarrow 
	\text{$SL^c$-oriented}\Rightarrow
	\text{$SL$-oriented}\Rightarrow
	\text{$Sp$-oriented}.
	\]
\end{lemma}
\begin{proof}
	We will show that an $\SLc$-oriented spectrum $A$ has a canonical normalized $\SL$-orientation, the other cases are similar. Let $\mcE$ be an $\SL$-bundle over $X\in\Smk$ and put $\thc(\mcE)=\thc(\Psi(\mcE))$ for the associated $\SLc$-bundle $\Psi(\mcE)$ (see Definition~\ref{def:Pfaffian}). Properties~\ref{def:orientation_iso} and~\ref{def:orientation_pull} of Definition~\ref{def:orientation} follow from the functoriality of $\Psi$, properties~\ref{def:orientation_mult} and~\ref{def:orientation_norm} follow from Lemma~\ref{lem:pfaf}.
\end{proof}

\begin{lemma} \label{lem:trivialThom}
	Let $A\in\SHk$ be a $G$-oriented spectrum and $X\in\Smk$. Then
	\[
	\thc(\triv^{G,n}_X)=\Sigma^n_{\T} 1 \in A^{0,0}(X;\triv^{G,n}_X ).
	\]
\end{lemma}
\begin{proof}
	For $X=S$ the claim follows from the properties~\ref{def:orientation_mult} and~\ref{def:orientation_norm} of Definition~\ref{def:orientation}. For a general $X$ the claim follows from $X=S$, properties~\ref{def:orientation_iso} and~\ref{def:orientation_pull} of Definition~\ref{def:orientation} and the canonical isomorphism $\triv^{G,n}_X\cong p^*\triv^{G,n}_S$ for the projection $p\colon X\to S$.
\end{proof}	

\begin{definition} \label{def:locally_standard}
	Let $A\in\SHk$ be a commutative ring spectrum and $\mcE$ be a rank $n$ vector $G$-bundle over $X\in\Smk$. We say that $\thc\in A^{0,0}(X;\mcE)$ is \textit{locally $G$-standard} if there exists a Zariski open cover $X=\bigcup_{i=1}^m U_i$ and isomorphisms of vector $G$-bundles $\theta_i\colon \mcE|_{U_i}\rightiso \triv^{G,n}_{U_i}$ such that $\thc|_{U_i}=\theta^A_{i} \Sigma_{\T}^n 1$. In the case of $G=\SLc$ following the convention of Definition~\ref{def:SLc} we abuse the notation writing $\theta^A_{i} \Sigma_{\T}^n 1$ where we should write $\underline{\theta}^A_{i} \Sigma_{\T}^n 1$ with $\theta_i=(\underline{\theta}_i,\hat{\theta}_i)\colon \mcE|_{U_i}\rightiso \triv^{\SLc,n}_{U_i}$ being the isomorphisms.
\end{definition}

\begin{lemma} \label{lem:pull_locally_standard}
	Let $A\in\SHk$ be a commutative ring spectrum, $\mcE$ be a rank $n$ vector $G$-bundle over $X\in\Smk$ and $\thc\in A^{0,0}(X;\mcE)$ be locally $G$-standard. Then for every morphism of smooth $S$-schemes $p\colon Y\to X$ the element $p^A(\thc) \in A^{0,0}(Y;\mcE)$ is locally $G$-standard.
\end{lemma}
\begin{proof}
	Let $X=\bigcup_{i=1}^m U_i$ be a Zariski open cover and $\theta_i\colon \mcE|_{U_i}\rightiso \triv^{G,n}_{U_i}$ be isomorphisms of vector $G$-bundles such that $\thc|_{U_i}=\theta^A_{i} \Sigma_{\T}^n 1$. Then for the open cover $Y=\bigcup_{i=1}^m p^{-1}(U_i)$ and isomorphisms $p^*\theta_i\colon p^*\mcE|_{p^{-1}(U_i)}\rightiso \triv^{G,n}_{p^{-1}(U_i)}$ one has $\thc|_{p^{-1}(U_i)}=(p^*\theta_i)^A \Sigma_{\T}^n 1$.
\end{proof}

\begin{lemma} \label{lem:local_Thom}
	Let $A\in\SHk$ be a commutative ring spectrum and $\mcE$ be a rank $n$ vector $G$-bundle over $X\in\Smk$. Suppose that $\thc\in A^{0,0}(X;\mcE)$ is locally $G$-standard. Then 
	\[
	-\cup \thc\colon A^{*,*}(-)\to A^{*,*} (-;\mcE)
	\]
	is an isomorphism of presheaves on the slice category $\Smk/X$.
\end{lemma}
\begin{proof}
	In view of Lemma~\ref{lem:pull_locally_standard} it is sufficient to check that
	\[
	-\cup \thc\colon A^{*,*}(X)\to A^{*,*} (X;\mcE)
	\]
	is an isomorphism. The proof follows via a standard Mayer-Vietoris argument. 
	
	We argue by induction on the minimal possible size of the Zariski open cover $X=\bigcup_{i=1}^m U_i$ satisfying $\thc|_{U_i}=\theta^A_{i} \Sigma_{\T}^n 1$ for some isomorphisms of vector $G$-bundles $\theta_i\colon \mcE|_{U_i}\rightiso \triv^{G,n}_{U_i}$. First suppose that $m=1$, i.e. there exists a trivialization $\theta\colon \mcE\rightiso \triv^{G,n}_X$. Then $-\cup \thc$ is the composition of the isomorphisms 
	\[
	\theta^A\circ \Sigma^n_\T \colon A^{*,*}(X)\rightiso A^{*,*}(X; \triv^{G, n}_X)\rightiso A^{*,*}(X;\mcE).
	\]
	
	For the inductive step choose a Zariski open cover $X=\bigcup_{i=1}^m U_i$ and isomorphisms of vector $G$-bundles $\theta_i\colon \mcE|_{U_i}\rightiso \triv^{G,n}_{U_i}$ such that $\thc|_{U_i}=\theta^A_{i} \Sigma_{\T}^n 1$. Put $V= \bigcup_{i=1}^{m-1} U_i$ and $W=U_m$. Consider the morphism of Mayer-Vietoris long exact sequences induced by $-\cup \thc$.
	\[
	\xymatrix @C=1.7pc{
		\hdots \ar[r] & A^{*,*}(X) \ar[r] \ar[d]^{\cup \thc} & A^{*,*}(V)\oplus A^{*,*}(W) \ar[r] \ar[d]^{(\cup \thc|_{V}) \oplus (\cup\thc|_{W})} & A^{*,*}(V\cap W) \ar[r] \ar[d]^{\cup\thc|_{V\cap W}} & \hdots \\
		\hdots \ar[r] & A^{*,*}(X;\mcE) \ar[r] & A^{*,*}(V;\mcE)\oplus A^{*,*}(W;\mcE) \ar[r] & A^{*,*}(V\cap W;\mcE) \ar[r]  & \hdots	
	}
	\]
	The elements $\thc|_{V}$, $\thc|_{W}$ and $\thc|_{V\cap W}$ are locally $G$-standard for the covers $V=\bigcup_{i=1}^{m-1} U_i$, $W=U_m$ and $V\cap W=(\bigcup_{i=1}^{m-1} U_i)\cap U_m$ respectively whence the second and the third vertical homomorphisms are isomorphisms by the inductive assumption. It follows from the five lemma that the first vertical homomorphism is an isomorphism as well.
\end{proof}

\begin{corollary} \label{cor:Thomiso}
	Let $A\in\SHk$ be a $G$-oriented spectrum. Then for a rank $n$ vector $G$-bundle $\mcE$  over $X\in\Smk$ the homomorphism
	\[
	-\cup \thc(\mcE)\colon A^{*,*}(-)\to A^{*,*}(-;\mcE)
	\]
	is an isomorphism of presheaves on the slice category $\Smk/X$.
\end{corollary}
\begin{proof}
	Applying Lemma~\ref{lem:locallytrivial} we may choose a covering $X=\bigcup_{i=1}^m U_i$ such that $\mcE|_{U_i}$ is trivial for every $i$. Lemma~\ref{lem:trivialThom} combined with the properties~\ref{def:orientation_iso} and~\ref{def:orientation_pull} of Definition~\ref{def:orientation} yields that $\thc(\mcE)\in A^{0,0}(X;\mcE)$ is locally $G$-standard for the chosen cover and for every choice of trivializations $\theta_i\colon \mcE|_{U_i}\rightiso \triv^{G,n}_{U_i}$. The claim follows by Lemma~\ref{lem:local_Thom}.
\end{proof}	
\begin{remark}
	A (non-normalized) \textit{$G$-orientation} of a commutative ring spectrum $A\in\SHk$ is a normalized $G$-orientation in the sense of Definition~\ref{def:orientation} with the property~\ref{def:orientation_norm} replaced with
	\begin{enumerate}
		\item[(4$'$)]	For every vector $G$-bundle $\mcE$  over $X\in\Smk$ the homomorphism
		\[
		-\cup \thc(\mcE)\colon A^{*,*}(X)\to A^{*,*}(X;\mcE)
		\]
		is an isomorphism.
	\end{enumerate}
	Corollary~\ref{cor:Thomiso} yields that a normalized $G$-orientation is a $G$-orientation in the above sense. On the other hand, a $G$-orientation with Thom classes $\thc(\mcE)$ gives rise to a normalized $G$-orientation as follows.
	
	$G=\GL,\SL$ or $\SLc$. Put $\alpha=\Sigma_\T^{-1}\thc(\triv^{G,1}_S)\in A^{0,0}(S)$. Property~{(4$'$)} yields that 
	\[
	A^{0,0}(S) \cup \alpha=A^{0,0}(S),
	\]
	i.e. that $\alpha$ is an invertible element of $A^{0,0}(S)$. Put $\widetilde{\thc}(\mcE)=\alpha^{-n} \cup \thc(\mcE)$ for a rank $n$ vector $G$-bundle $\mcE$. This rule defines a  normalized $G$-orientation.
	
	$G=\Sp$. Put $\alpha=\Sigma_\T^{-2}\thc(\triv^{G,2}_S)\in A^{0,0}(S)$. As above, property~{(4$'$)} yields that $\alpha$ is invertible in $A^{0,0}(S)$. Put $\widetilde{\thc}(\mcE)=\alpha^{-n} \cup \thc(\mcE)$ for a rank $2n$ symplectic bundle $\mcE$. This rule defines a normalized symplectic orientation.
\end{remark}

\section{$\mathbf{SL^c}$ Thom isomorphisms for $\mathbf{SL}$-oriented cohomology} \label{section:SLcSL}

\begin{lemma}[{cf. \cite[Lemma~2]{An16b}}] \label{lem:dualThom}
	Let $E$ be a vector bundle over $X\in \Smk$ and $L$ be a line bundle over $X$. Then there exists an isomorphism 
	\[
	\rho\colon \Th(E\oplus L)\rightiso \Th(E\oplus L^\vee)
	\]
	in the motivic unstable homotopy category $\mathcal{H}_\bullet (S)$.
\end{lemma}
\begin{proof}
	Let $E^o$, $L^o$ and $(L^\vee)^o$ be the complements to the zero sections of the respective bundles. The Thom spaces $\Th(E\oplus L)$ and $\Th(E\oplus L^\vee)$ can be realized as the total cofibers of the following diagrams.
	\[
	\xymatrix{
		E^o\times_X L^o \ar[r] \ar[d] & L^o \ar[d]\\
		E^o \ar[r] & X
	}
	\qquad
	\xymatrix{
		E^o\times_X (L^\vee)^o \ar[r] \ar[d] & (L^\vee)^o \ar[d]\\
		E^o \ar[r] & X
	}
	\]
	Here all the morphisms are the projections. There exists a unique isomorphism of $X$-schemes $L^o\to (L^\vee)^o$ that on the fibers takes a vector $v$ to the functional $f$ such that $f(v)=1$. This isomorphism gives rise to a morphism of the above diagrams inducing the isomorphism $\rho$ on the total cofibers.
\end{proof}

\begin{remark}
	It is clear that the isomorphism $\rho=\rho_{X,E,L}$ constructed in Lemma~\ref{lem:dualThom} is functorial in $X$, i.e. that for $Y\in\Smk$ and a regular morphism $p\colon Y\to X$ the following diagram commutes.
	\[
	\xymatrix @C=4.5pc{
		Th(p^*E\oplus p^*L)\ar[r]^{\rho_{Y,p^*E,p^*L}} \ar[d]_{p_{E\oplus L}}  & \Th(p^*E\oplus p^*L^\vee)  \ar[d]^{p_{E\oplus L^\vee}} \\
		Th(E\oplus L)\ar[r]^{\rho_{X,E,L}} & \Th(E\oplus L^\vee)
	}
	\]
	Here $p_{E\oplus L}$ and $p_{E\oplus L^\vee}$ are induced by the canonical morphisms of total spaces $p^*E\oplus p^*L \to E\oplus L$ and $p^*E\oplus p^*L^\vee\to E\oplus L^\vee$.
\end{remark}

\begin{theorem} \label{thm:SLvsSLc}
	Let $A\in\SHk$ be an $SL$-oriented spectrum and $\mcE=(E,L,\lambda)$ be a rank $n$ vector $SL^c$-bundle over $X\in\Smk$. Then there exists an isomorphism 
	\[
	A^{*,*}(-)\rightiso A^{*,*}(-;E)
	\]
	of presheaves on the slice category $\Smk/X$.
\end{theorem}
\begin{proof}
	The isomorphism
	$
	A^{*,*}(X)\rightiso A^{*,*}(X;E)
	$	
	is given by the following chain of isomorphisms:
	\[
	A^{*,*}(X)\xrightarrow{\phi_1} A^{*,*}(X;E\oplus L^\vee\oplus L^\vee) \xrightarrow{\psi} A^{*,*}(X;E\oplus L^\vee\oplus L) \xrightarrow{\phi_2} A^{*,*}(X;E).
	\]
	Here
	\begin{itemize}
		\item 
		$\phi_1=-\cup \thc(\widetilde{\mcE})$ is given by the cup-product with the Thom class for the rank $n+2$ vector $\SL$-bundle $\widetilde{\mcE}=(E\oplus L^\vee\oplus L^\vee, \widetilde{\lambda})$. The isomorphism $\widetilde{\lambda}\colon \det(E\oplus L^\vee\oplus L^\vee)\rightiso \triv_Y$ is given by the composition of the isomorphisms
		$
		\det(E\oplus L^\vee\oplus L^\vee)\rightiso \det E\otimes (L^\vee)^{\otimes 2}\xrightarrow{\lambda\otimes \id} L^{\otimes 2} \otimes (L^\vee)^{\otimes 2} \rightiso \triv_X.
		$
		\item 
		$\psi=\rho^A$ is the pullback along the isomorphism 
		\[
		\rho\colon \Th(E\oplus L^\vee\oplus L)\rightiso \Th(E\oplus L^\vee\oplus L^\vee)
		\]
		given by Lemma~\ref{lem:dualThom}.
		\item
		$\phi_2=\widetilde{\phi}_2^{-1}$ is the inverse to the isomorphism
		\[
		\widetilde{\phi}_2\colon A^{*,*}(X;E)\xrightarrow{-\cup \thc(L^\vee\oplus L,\lambda_{ev})} A^{*,*}(X;E\oplus L^\vee\oplus L)
		\]
		given by the cup-product with the Thom class for the rank $2$ vector $\SL$-bundle $(L^\vee\oplus L,\lambda_{ev})$ where $\lambda_{ev}\colon \det (L^\vee\oplus L)\rightiso L^\vee\otimes L\xrightarrow{ev} \triv_X$ is induced by the evaluation homomorphism. Note that $\widetilde{\phi}_2$ is an isomorphism by \cite[Lemma~3]{An16b}.
	\end{itemize}
	It is clear that the construction is functorial in $X$ whence the claim.
\end{proof}

\begin{remark}
	Let $A\in\SHk$ be an $\SL$-oriented spectrum. Theorem~\ref{thm:SLvsSLc} allows one to define a Thom class for an $\SLc$-bundle $\mcE$ over $X\in \Smk$ as the image of $1\in A^{0,0}(X)$ under the isomorphism $A^{0,0}(X)\rightiso A^{0,0}(X;\mcE)$. A natural question is whether this rule gives rise to an $\SLc$-orientation of $A$ and if so, what are the relations between this $\SLc$-orientation and the original $\SL$-orientation. Unfortunately, while functoriality and normalization (properties \ref{def:orientation_iso}, \ref{def:orientation_pull} and \ref{def:orientation_norm} of Definition~\ref{def:orientation}) are straightforward, multiplicativity (property \ref{def:orientation_mult} of Definition~\ref{def:orientation}) seems to be more involved and at the moment it is not clear whether this property holds. We leave these questions to future investigations.
\end{remark}

\section{Orienting $\mathbf{A}$ when $\mathbf{A^{0,0}(-)}$ is a sheaf} \label{section:orientsheaf}

\begin{lemma} \label{lem:orient_unique}
	Let $A\in\SHk$ be a commutative ring spectrum and suppose that $A^{0,0}(-)$ is a sheaf in the Zariski topology. Then $A$ admits at most one normalized $G$-orientation.
\end{lemma}
\begin{proof}
	Suppose that $A$ admits normalized $G$-orientations with Thom classes $\thc(\mcE)$ and $\widetilde{\thc}(\mcE)$. Corollary~\ref{cor:Thomiso} yields that for a rank $n$ vector $G$-bundle $\mcE$ over $X\in\Smk$ the homomorphism
	\[
	-\cup \thc(\mcE)\colon A^{0,0}(-)\to A^{0,0}(-;\mcE)
	\]
	is an isomorphism of presheaves on the slice category $\Smk/X$ whence $A^{0,0}(-;\mcE)$ is a sheaf on the small Zariski site of $X$. Apply Lemma~\ref{lem:locallytrivial} and choose a cover $X=\bigcup_{i=1}^m U_i$ together with trivializations $\theta_i\colon \mcE|_{U_i}\rightiso \triv^{G,n}_{U_i}$. Lemma~\ref{lem:trivialThom} combined with the properties~\ref{def:orientation_iso} and~\ref{def:orientation_pull} of Definition~\ref{def:orientation} yields that 
	\[
	\thc(\mcE)|_{U_i}=\thc(\mcE|_{U_i})=\theta_i^A\Sigma_\T^n 1= \widetilde{\thc}(\mcE|_{U_i})=\widetilde{\thc}(\mcE)|_{U_i}
	\]
	for every $i$. Thus $\thc(\mcE)$ and $\widetilde{\thc}(\mcE)$ coincide locally. We have already shown that $A^{0,0}(-;\mcE)$ is a sheaf whence $\thc(\mcE)=\widetilde{\thc}(\mcE)$
\end{proof}

\begin{lemma} \label{lem:unique_local}
	Let $\mcE$ be a rank $n$ vector $SL^c$-bundle over $X\in\Smk$, let $A\in \SHk$ be a commutative ring spectrum and suppose that $A^{0,0}(-)$ restricts to a sheaf on the small Zariski site of $X$. Then there exists at most one locally $SL^c$-standard element in $A^{0,0}(X;\mcE)$.
\end{lemma}
\begin{proof}
	Let $\thc,\widetilde{\thc}\in A^{0,0}(X;\mcE)$ be locally $\SLc$-standard elements. Lemma~\ref{lem:local_Thom} yields that 
	\[
	-\cup \thc \colon A^{0,0}(-) \to A^{0,0}(-;\mcE)
	\]
	is an isomorphism of presheaves on the slice category $\Smk/X$. Then $A^{0,0}(-;\mcE)$ restricts to a sheaf on the small Zariski site of $X$.
	
	Choose the covers and trivializations from Definition~\ref{def:locally_standard} for $\thc$ and $\widetilde{\thc}$. Intersecting the elements of the covers and restricting trivializations we may choose a Zariski open cover $X=\bigcup_{i=1}^m U_i$ and trivializations $\theta_i,\widetilde{\theta}_i\colon \mcE_{U_i}\rightiso \triv^{\SLc,n}_{U_i}$ such that
	\[
	\thc|_{U_i}=\theta_i^A \Sigma_{\T}^n 1,\quad \widetilde{\thc}|_{U_i}=\widetilde{\theta}_i^A \Sigma_{\T}^n 1.
	\]
	The automorphisms
	\[
	\widetilde{\theta}_i\circ \theta_i^{-1}\colon \triv^{\SLc,n}_{U_i}\rightiso \triv^{\SLc,n}_{U_i}
	\]
	are given by 
	\[
	(g_i,\lambda_i)\in \mathrm{SL^c_n}(R_i)\subseteq \GL_n(R_i)\times R_i^*, \quad R_i=\Gamma(U_i,\struct_{U_i})
	\]
	with $\det g_i=\lambda_i^2$. Put $g_i'=g_i d_i$ for the diagonal matrix $d_i=\operatorname{diag}(\lambda_i^{-2},1,\hdots 1)$. Then $\det g_i'=1$. Every matrix of determinant $1$ over a local ring is an elementary matrix (i.e. a product of transvections) whence for every $i$ there exists a Zariski open covering $U_i=\bigcup_{j=1}^{m_i} U_{ij}$ such that $g'_i|_{U_{ij}}$ is an elementary matrix for every $j$. It is well known that the action of an elementary matrix on the Thom space of a trivial vector bundle is homotopy equivalent to the identity morphism (see, for example, \cite[Lemma~1]{An16b}). Moreover,  multiplication by a perfect square on the Thom space of a trivial line bundle is also homotopy equivalent to the identity morphism, see, for example, \cite[Lemma~5]{An16b} (the standing assumption of the loc. cit. that the base scheme $S$ is the spectrum of a field is not used in the proof of the lemma). Hence
	\begin{multline*}
	\widetilde{\thc}|_{U_{ij}}=(\widetilde{\theta}_i^A\Sigma_{\T}^n 1)|_{U_{ij}} = (\widetilde{\theta}_i|_{U_{ij}})^A(\Sigma_{\T}^n 1) = (\theta_i|_{U_{ij}})^A\circ (\widetilde{\theta_i}\circ\theta_i^{-1})|_{U_{ij}}^A (\Sigma_{\T}^n 1) = \\
	= (\theta_i|_{U_{ij}})^A\circ (d_i|_{U_{ij}})^A \circ (g'_i|_{U_{ij}})^A (\Sigma_{\T}^n 1) = (\theta_i|_{U_{ij}})^A (\Sigma_{\T}^n 1) =(\theta_i^A\Sigma_{\T}^n 1)|_{U_{ij}}=\thc|_{U_{ij}}.
	\end{multline*}
	Then $\thc$ and $\widetilde{\thc}$ coincide locally and since we have already shown that $A^{0,0}(-;\mcE)$ is a sheaf this yields $\thc=\widetilde{\thc}$.
\end{proof}

\begin{theorem} \label{thm:sheafSL}
	Let $A\in \SHk$ be a commutative ring spectrum and suppose that $A^{0,0}(-)$ is a sheaf in the Zariski topology. Then $A$ admits a unique normalized $SL^c$-orientation.
\end{theorem}
\begin{proof}
	The proof goes in two steps: first we show that for every rank $n$ vector $\SLc$-bundle $\mcE$ there exists a locally $\SLc$-standard element $\thc(\mcE)\in A^{0,0}(X;\mcE)$ and then we apply the uniqueness of such elements given by Lemma~\ref{lem:unique_local} in order to show that they satisfy the properties of Definition~\ref{def:orientation}.
	
	Let $\mcE$ be a rank $n$ vector $\SLc$-bundle over $X\in \Smk$. Applying Lemma~\ref{lem:locallytrivial} choose a Zariski open cover $X=\bigcup_{i=1}^m U_i$ and isomorphisms of vector $\SLc$-bundles $\theta_i\colon \mcE|_{U_i}\rightiso \triv^{\SLc,n}_{U_i}$. For every $i$ put 
	\[
	\thc_{i}=\theta_i^A\Sigma_{\T}^n 1\in A^{0,0}(U_i;\mcE).
	\]
	Denote $U_{\le j}=\bigcup_{i=1}^j U_i$. We will inductively construct locally $\SLc$-standard elements $\thc_{\le j}\in A^{0,0}(U_{\le j};\mcE)$. Put $\thc_{\le 1}=\thc_1$. In order to construct $\thc_{\le j}$ consider the following fragment of the Mayer-Vietoris long exact sequence.
	\[
	\hdots \to A^{0,0}(U_{\le j};\mcE)\to A^{0,0}(U_{\le j-1};\mcE) \oplus A^{0,0}(U_{j};\mcE) \to A^{0,0}(U_{\le j-1}\cap U_j;\mcE)\to \hdots
	\]
	Lemma~\ref{lem:pull_locally_standard} yields that both the elements 
	\[
	\thc_{\le j-1}|_{U_{\le j-1}\cap U_j},\,\thc_{j}|_{U_{\le j-1}\cap U_j}\in A^{0,0}(U_{\le j-1}\cap U_j;\mcE)
	\]
	are locally $\SLc$-standard whence Lemma~\ref{lem:unique_local} yields
	\[
	\thc_{\le j-1}|_{U_{\le j-1}\cap U_j}=\thc_{j}|_{U_{\le j-1}\cap U_j}.
	\]
	It follows from the above Mayer-Vietoris long exact sequence that there exists some
	\[
	\thc_{\le j}\in A^{0,0}(U_{\le j};\mcE)
	\]
	such that $\thc_{\le j}|_{U_{\le j-1}}=\thc_{\le j-1}$ and $\thc_{\le j}|_{U_j}=\thc_j$. This element is clearly locally $\SLc$-standard for the open cover $U_{\le j}=\bigcup_{i=1}^j U_i$ and isomorphisms of vector $\SLc$-bundles $\theta_i\colon \mcE|_{U_i}\rightiso \triv^{\SLc,n}_{U_i}$. Put 
	\[
	\thc(\mcE)=\thc_{\le m}.
	\]
	Note that $\thc(\mcE)$ does not depend on the choices made above, i.e. on the choice of an open cover  $X=\bigcup_{i=1}^m U_i$ and isomorphisms $\theta_i$, since $\thc(\mcE)$ is locally $\SLc$-standard by construction whence unique by Lemma~\ref{lem:unique_local}.
	
	Now we are going to check that the properties of Definition~\ref{def:orientation} hold for the constructed elements $\thc(\mcE)$. Each time it is sufficient to check that the element at right-hand side of the corresponding equality is locally $\SLc$-standard and then apply Lemma~\ref{lem:unique_local}. 
	\begin{enumerate}
		\item 
		For the property~\ref{def:orientation_iso} choose a Zariski open cover $X=\bigcup_{i=1}^m U_i$ and isomorphisms of vector $\SLc$-bundles $\theta_i\colon \mcE'|_{U_i}\rightiso \triv^{\SLc,n}_{U_i}$ such that $\thc(\mcE')|_{U_i}=\theta_i^A\Sigma^n_\T 1$. Then $\theta^A\thc(\mcE')_{U_i}=(\theta_i\circ \theta|_{U_i})^A\Sigma^n_\T 1$ whence $\theta^A\thc(\mcE')$ is locally $\SLc$-standard with the cover $X=\bigcup_{i=1}^m U_i$ and trivializations $\theta_i\circ \theta|_{U_i}$.
		\item 
		For the property~\ref{def:orientation_pull} note that the element $f^A\thc(\mcE)$ is locally $\SLc$-standard by Lemma~\ref{lem:pull_locally_standard}.
		\item 
		For the property~\ref{def:orientation_mult} intersecting the elements of open covers for $\mcE$ and $\mcE'$ and restricting isomorphisms we may choose a Zariski open cover $X=\bigcup_{i=1}^m U_i$ and isomorphisms of vector $\SLc$-bundles 
		\[
		\theta_i\colon \mcE|_{U_i}\rightiso \triv^{\SLc,n}_{U_i},\quad \widetilde{\theta}_i\colon \mcE'|_{U_i}\rightiso \triv^{\SLc,n'}_{U_i}\]
		such that $\thc(\mcE)|_{U_i}=\theta_i^A\Sigma^n_\T 1$, $\thc(\mcE')|_{U_i}=\widetilde{\theta}_i^A\Sigma^{n'}_\T 1$. Then 
		\[
		(q_1^A \thc(\mathcal{E})\cup q_2^A \thc(\mathcal{E}'))|_{U_i}=(\theta_i\oplus \widetilde{\theta}_i)^A (\Sigma^{n+n'}_\T 1)
		\]
		whence $q_1^A \thc(\mathcal{E})\cup q_2^A \thc(\mathcal{E}')$ is locally $\SLc$-standard with the open cover $X=\bigcup_{i=1}^m U_i$ and trivializations $\theta_i\oplus \widetilde{\theta}_i$.
		\item 
		For the property~\ref{def:orientation_norm} note that $\Sigma_{\T}^n 1$ is tautologically locally $\SLc$-standard.
	\end{enumerate}
	
	The uniqueness of the normalized $\SLc$-orientation follows from Lemma~\ref{lem:orient_unique}.
\end{proof}
\begin{corollary} \label{cor:sheafSL}
	Let $A\in \SHk$ be a commutative ring spectrum and suppose that $A^{0,0}(-)$ is a sheaf in the Zariski topology. Then $A$ admits a unique normalized $SL$-orientation.	
\end{corollary}
\begin{proof}
	The existence follows from Lemma~\ref{lem:orientationshomo} and Theorem~\ref{thm:sheafSL}, the uniqueness follows from Lemma~\ref{lem:orient_unique}.
\end{proof}

\begin{remark} \label{rem:asheaf}
	Let $S=\Spec k$ be the spectrum of a field. Theorem~\ref{thm:sheafSL} and Corollary~\ref{cor:sheafSL} provide a uniform approach to the $\SLc$ and $\SL$ orientations of some well known motivic cohomology theories that are known to be $\SLc$-oriented by some other means. In particular, $A^{0,0}(-)$ is a sheaf in the following cases.
	\begin{enumerate}
		\item 
		$A$ is the Eilenberg-MacLane spectrum for a homotopy module \cite[Section~5.2]{Mor04}, i.e. $A$ belongs to the heart of the homotopy $t$-structure on the category $\SH(k)$.
		\item
		$A=\mathrm{H}\Z$ is the spectrum representing motivic cohomology. In this case $A^{0,0}(-)=\mathrm{H}\Z^{0,0}(-)$ is well known to be the constant sheaf $\underline{\Z}$.
		\item
		$A=\mathrm{H_{MW}}\Z$ is the spectrum representing Milnor-Witt motivic cohomology \cite{DF17}, $k$ is infinite and perfect. Then $A^{0,0}(-)=\mathrm{H_{MW}}\Z^{0,0}(-)\cong \underline{\mathrm{GW}}$ is the unramified Grotendieck--Witt sheaf \cite[Section~3.2.1]{DF17}.
	\end{enumerate}
\end{remark}

\section{Interlude on the Hopf element and the connecting homomorphism} \label{section:hopf}
\begin{definition}[{\cite[Definition~3.6]{An16a}}]
	Let $A\in\SHk$ be a commutative ring spectrum and $u\in \Gamma(X,\struct^*_X)$ be an invertible regular function on $X\in\Smk$. Put
	\[
	\langle u \rangle = \langle u \rangle_A = \Sigma^{-1}_\T f_u^A \Sigma_\T 1 \in A^{0,0}(X)
	\]
	for the morphism $f_u\colon X_+\wedge \T\to X_+\wedge \T$ given by $f_u(x,t)=(x,u(x)t)$. 
\end{definition}

\begin{lemma} \label{lem:quadratic}
	Let $u,v\in \Gamma(X,\struct^*_X)$ be invertible regular functions on $X\in\Smk$. Then for a commutative ring spectrum $A\in\SHk$ one has
	\[
	\langle v^2u\rangle =\langle u\rangle \in A^{0,0}(X).
	\]
\end{lemma}
\begin{proof}
	The same proof as in \cite[Lemma~5]{An16b} yields that 
	\[
	\Sigma_{\T}^\infty f_u= \Sigma_{\T}^\infty f_{v^2u} \colon \Sigma_{\T}^\infty X_+\wedge \T \to \Sigma_{\T}^\infty X_+\wedge \T
	\]
	whence the claim. Note that the assumption $S=\Spec k$ of \cite{An16b} is not needed for the proof, one applies literally the same reasoning over a general base.
\end{proof}

\begin{definition}
	The \textit{Hopf map} is the morphism $H\colon \A^2-(0,0) \to \PP^1$ given by $H(x,y)=[x:y]$. Recall that there are canonical isomorphisms 
	\[
	(\A^2-(0,0),(1,1))\cong \Sph^{3,2},\quad (\PP^1,[1:1])\cong \Sph^{2,1}
	\]
	in the unstable homotopy category \cite[Lemma~2.15, Example~2.20]{MV99} whence the Hopf map gives rise to the \textit{Hopf element} $\eta \in \pi^{-1,-1}(S)$. See \cite[Definition~3.5]{An16a} for the precise formula defining $\eta$ out of $H$ that we are going to use in the current paper. For a commutative ring spectrum $A\in\SHk$ we abuse the notation and denote by the same letter the associated Hopf element $\eta \in A^{-1,-1}(S)$.
\end{definition}

\begin{lemma} \label{lem:connecting}
	Let $X\in\Smk$ and $t\in \Gamma(X\times \Gm,\struct^*_{X\times \Gm})$ be the invertible function on $X\times \Gm$ given by the coordinate function on $\Gm$. Then for a commutative ring spectrum $A\in\SHk$ 
	one has
	\[
	\partial \langle -t\rangle = \Sigma_\T \eta
	\]
	where $\partial$ is the connecting homomorphism in the localization long exact sequence
	\[
	\cdots \to A^{*,*}(X_+\wedge \T) \to A^{*,*}(X\times \A^1 ) \to A^{*,*}(X\times \Gm) \xrightarrow{\partial} A^{*+1,*}(X_+\wedge \T)\to\cdots
	\]
\end{lemma}
\begin{proof}
	Lemma~\ref{lem:quadratic} yields that $\langle -t\rangle = \langle -t^{-1}\rangle$ Then the claim follows from \cite[Theorem~3.8]{An16a}. Note that the assumption $S=\Spec k$ of \cite{An16a} is not needed for the proof, one applies literally the same reasoning over a general base. 
\end{proof}

\begin{remark}
	One should not be too surprised to see $\langle -t\rangle$ in the statement of the above lemma since the computation depends on the particular choice of isomorphisms $(\A^2-(0,0),(1,1))\cong \Sph^{3,2}$ and $(\PP^1,[1:1])\cong \Sph^{2,1}$ defining $\eta$. In the current paper we adopted the choices made in \cite[Definition~3.5]{An16a}, if one defines $\widetilde{\eta}$ via the same isomorphism $(\A^2-\{0,0\},(1,1))\cong \Sph^{3,2}$ and the same isomorphism $ (\PP^1,[1:1])\cong \Sph^{2,1}$ precomposed with an automorphism $(\PP^1,[1:1])\to (\PP^1,[1:1])$ given by $[x:y]\mapsto [y:x]$ then one can show that $	\partial \langle t\rangle = \Sigma_\T \widetilde{\eta}$.
\end{remark}

\section{Characteristic classes and Hopf element} \label{section:charclasses}

\begin{definition}
	Let $A\in\SHk$ be a $G$-oriented spectrum and $\mcE$ be a rank $n$ vector $G$-bundle over $X\in\Smk$. Put
	\[
	e(\mcE)=z^A \thc(\mcE) \in A^{2n,n}(X)
	\]
	for the pullback 
	\[
	z^A\colon A^{0,0}(X;\mcE)= A^{2n,n}(\Th(\mcE)) \to A^{2n,n}(X)
	\]
	along the morphism $z\colon X\to \Th(\mcE)$ induced by the zero section of the bundle $\mcE$. We refer to $e(\mcE)$ as the \textit{Euler class} of $\mcE$.
\end{definition}

\begin{lemma}
	Let $A\in\SHk$ be a $G$-oriented spectrum. Then
	\begin{enumerate}
		\item 
		$e(\mcE)=e(\mcE')$ for isomorphic vector $G$-bundles $\mcE$ and $\mcE'$.
		\item 
		$f^Ae(\mcE)=e(f^*\mcE)$ for a morphism of smooth $S$-schemes $f\colon Y\to X$ and a vector $G$-bundle $\mcE$ over $X$.
		\item 
		$e(\mcE\oplus \mcE')=e(\mcE)\cup e(\mcE')$ for vector $G$-bundles $\mcE$ and $\mcE'$ over $X\in \Smk$.
	\end{enumerate}
\end{lemma}
\begin{proof}
	Straightforward from the properties stated in Definition~\ref{def:orientation}.
\end{proof}

\begin{lemma} \label{lem:change_orient}
	Let $A\in \SHk$ be an $SL$-oriented spectrum and $(E,\lambda)$ be a rank~$n$ vector $SL$-bundle over $X\in \Smk$. Then for $u\in \Gamma(X,\struct_X^*)$ one has
	\[
	\thc(E,u \cdot \lambda) = \langle u \rangle \cup \thc(E,\lambda) \in A^{0,0}(X;E).
	\]
\end{lemma}
\begin{proof}
	Property~\ref{def:orientation_mult} of Definition~\ref{def:orientation} yields 
	\[
	\thc(E,u\cdot\lambda) \cup \thc(\triv_X,\id) =\thc(E\oplus \triv_X,u\cdot\lambda\otimes \id) = \thc(E,\lambda) \cup \thc(\triv_X,u\cdot \id).
	\]
	It follows from Lemma~\ref{lem:trivialThom} that $\thc(\triv_X,\id)=\Sigma_{\T} 1$ whence
	\[
	\thc(E,u\cdot\lambda)= \thc(E,\lambda) \cup \Sigma^{-1}_\T \thc(\triv_X,u\cdot \id) = \thc(E,\lambda)\cup  \langle u \rangle. \qedhere
	\]
	
\end{proof}

\begin{theorem}[{cf. \cite[Proposition~7.2]{Lev17}}] \label{thm:Euler_torsion}
	Let $A\in \SHk$ be an $SL$-oriented spectrum and $\mcE=(E,\lambda)$ be a vector $SL$-bundle over $X\in \Smk$. Suppose that there exists an isomorphism of vector bundles $E\cong E_1\oplus E_2$ with $E_1$ being of odd rank. Then $\eta \cup e(\mcE)=0$.
\end{theorem}
\begin{proof}
	Choose an isomorphism $\theta\colon E_1\oplus E_2 \rightiso  E$. Without loss of generality we may assume $E=E_1\oplus E_2$, so we omit $\theta$ from notation. Let $\pi\colon X\times \Gm\to X$ be the projection and denote $t\in \Gamma(X\times \Gm,\struct^*_{X\times \Gm})$ the invertible function on $X\times \Gm$ given by the coordinate function on $\Gm$. Consider the automorphism $\rho\colon \pi^*E_1\to \pi^* E_1$ induced by the multiplication by $-t$, i.e. $\rho(v)=-tv$ on the sections. Then $\rho$ gives rise to an isomorphism of vector $\SL$-bundles 
	\[
	(\rho\oplus \id) \colon (\pi^*(E_1\oplus E_2),-t^{2n+1}\cdot \pi^*\lambda) \rightiso (\pi^*(E_1\oplus E_2),\pi^*\lambda)
	\]
	with $2n+1=\rank E_1$. Property~\ref{def:orientation_iso} of Definition~\ref{def:orientation} yields
	\[
	\thc(\pi^*(E_1\oplus E_2),-t^{2n+1}\cdot \pi^*\lambda)
	=(\rho\oplus \id)^A\thc (\pi^*(E_1\oplus E_2),\pi^*\lambda).
	\]
	Then it follows from Lemmas~\ref{lem:change_orient} and~\ref{lem:quadratic} that 
	\begin{multline*}
	(\rho\oplus \id)^A\thc (\pi^*(E_1\oplus E_2),\pi^*\lambda) = \langle -t^{2n+1}\rangle \cup \thc (\pi^*(E_1\oplus E_2),\pi^*\lambda) = \\
	= \langle -t\rangle \cup \thc (\pi^*(E_1\oplus E_2),\pi^*\lambda).
	\end{multline*}
	For the morphism $z\colon X\times \Gm\to \Th(\pi^*(E_1\oplus E_2))$ induced by the zero section we have $(\rho\oplus \id)\circ z=z$ whence 
	\begin{multline*}
	\pi^A e(E_1\oplus E_2,\lambda)=e(\pi^*(E_1\oplus E_2),\pi^*\lambda)  =z^A\thc (\pi^*(E_1\oplus E_2),\pi^*\lambda) =\\
	= z^A (\rho\oplus \id)^A \thc (\pi^*(E_1\oplus E_2),\pi^*\lambda)
	= z^A(\langle -t \rangle \cup \thc (\pi^*(E_1\oplus E_2),\pi^*\lambda)) =\\
	= \langle -t \rangle \cup z^A(\thc (\pi^*(E_1\oplus E_2),\pi^*\lambda))= \langle -t \rangle \cup\pi^A e(E_1\oplus E_2,\lambda).	
	\end{multline*}
	
	Let $\partial$ be the connecting homomorphism in the localization long exact sequence
	\[
	\cdots \to A^{*,*}(X_+\wedge \T) \to A^{*,*}(X\times \A^1 ) \to A^{*,*}(X\times \Gm) \xrightarrow{\partial} A^{*+1,*}(X_+\wedge \T)\to\cdots
	\]
	Then for every $\alpha\in A^{*,*}(X)$ one has $\partial (\pi^A\alpha) =0$. Lemma~\ref{lem:connecting} combined with the above yields 
	\begin{multline*}
	0=\partial(\pi^A e(E_1\oplus E_2,\lambda))=\partial ( \langle -t \rangle \cup  \pi^A e (E_1\oplus E_2,\lambda))=\\
	=  \partial(\langle -t \rangle) \cup  e (E_1\oplus E_2,\lambda) = \eta \cup e (E_1\oplus E_2,\lambda). \qedhere
	\end{multline*}
\end{proof}

\begin{remark}
	Theorem~\ref{thm:Euler_torsion} improves \cite[Corollary~2]{An15} where it was proved that for a vector $\SL$-bundle of odd rank one has $e(\mcE)=0$ after inverting $\eta$.
\end{remark}

\begin{definition}[{\cite[Definition 14.1]{PW10a}}] \label{def:Borel}
	Let $A\in\SHk$ be a commutative ring spectrum. A theory of \textit{Borel classes} on $A$ is a rule which assigns to every symplectic bundle $\mcE$ over $X\in\Smk$ a sequence of elements 
	\[
	b_i(\mcE)=b^A_i(\mcE)\in A^{4i,2i}(X),\quad i\ge 1,
	\]
	with the following properties:
	\begin{enumerate}
		\item
		For isomorphic symplectic bundles $\mcE\cong \mcE'$ one has $b_i(\mcE)=b_i(\mcE')$ for all $i$.
		\item
		For a morphism of smooth $S$-schemes $f\colon Y\to X$ and a symplectic bundle $\mcE$ over $X$ one has $f^Ab_i(\mcE)=b_i(f^*\mcE)$ for all $i$.
		\item
		For $X\in\Smk$ the homomorphism 
		\[
		A^{*,*}(X)\oplus A^{*-4,*-2}(X) \to A^{*,*}(\HProj^1\times_S X)
		\]
		given by $a+a'\mapsto \pi_2^A(a)+\pi_2^A(a')\cup \pi_1^A(b_1(\mathcal{T}))$ is an isomorphism. Here 
		\begin{itemize}
			\item 
			$\HProj^1=\Sp_4/(\Sp_2\times \Sp_2)$ is the quaternionic projective line that is the variety of nondegenerate symplectic planes in the $4$-dimensional symplectic space $\hyperb(\triv^{\oplus 2}_S)$.
			\item 
			$\mathcal{T}$ is the tautological rank $2$ symplectic bundle over $\HProj^1$.
			\item
			$\pi_1\colon \HProj^1\times_S X\to \HProj^1,\, \pi_2\colon \HProj^1\times_S X\to X$ are the projections.			
		\end{itemize}
		\item
		For the trivialized rank $2$ symplectic bundle $\hyperb(\triv_S)=\triv^{\Sp,2}_S$ over $S$ one has
		$b_1(\hyperb(\triv_S))=0\in A^{4,2}(S)$.
		\item
		For a rank $2n$ symplectic bundle $\mcE$ one has $b_i(\mcE)=0$ for $i>n$.
		\item
		For symplectic bundles $\mcE,\mcE'$ over $X$ one has $b_t(\mcE)b_t(\mcE')=b_t(\mcE\oplus \mcE')$,
		where 
		\[
		b_t(\mcE)=1+b_1(\mcE)t+b_2(\mcE)t^2+\dots\in A^{*,*}(X)[t].
		\]
	\end{enumerate}
	We refer to $b_i(E)$ as \textit{Borel classes} of $E$ and $b_t(E)$ is the \textit{total Borel class}. 
	
	It follows from \cite[Theorem~14.4]{PW10a} that a symplectically oriented spectrum $A\in\SHk$ admits a unique theory of Borel classes such that $b_{n}(\mcE)=e(\mcE)$ for every rank $2n$ symplectic bundle $\mcE$ over $X\in\Smk$. Moreover, Lemma~\ref{lem:orientationshomo} yields that an $\SL$-oriented spectrum is symplectically oriented in a canonical way. Thus every $\SL$-oriented spectrum admits a canonical theory of Borel classes such that for every rank $2n$ symplectic bundle $\mcE$ one has
	\[
	b_{n}(\mcE)=e(\Phi(\mcE))
	\]
	with $\Phi(\mcE)$ being the associated rank $2n$ vector $\SL$-bundle (see Definition~\ref{def:SLbundle}).
\end{definition}

\begin{definition}[{cf. \cite[Definition~7]{An15} and \cite[Definition~5.6]{HW17}}] \label{def:Pontryagin}
	Let $A\in\SHk$ be a commutative ring spectrum with a chosen Borel classes theory (e.g. a symplectically or $\SL$-oriented spectrum). For a rank $n$ vector bundle $E$ over $X\in\Smk$ put
	\[
	p_i(E)=(-1)^i b_{2i} (\hyperb(E)) \in A^{8i,4i}(X),\quad i= 1,2,\ldots,
	\]
	for the associated rank $2n$ hyperbolic symplectic bundle $\hyperb(E)$. We refer to $p_i(E)$ as \textit{Pontryagin classes} of $E$ and 
	\[
	p_t(E)=1+p_1(E)t+p_2(E)t^2+\dots\in A^{*,*}(X)[t]
	\]
	is the \textit{total Pontryagin class}.
\end{definition}

\begin{lemma} \label{lem:torsionBorel}
	Let $A\in \SHk$ be an $SL$-oriented spectrum and $E$ be a rank $2n+1$ vector bundle over $X\in \Smk$. Then $\eta \cup b_{2n+1}(\hyperb(E))=0$.
\end{lemma}
\begin{proof}
	Let $\Psi(\hyperb(E))=(E\oplus E^\vee,\lambda)$ be the rank $4n+2$ vector $\SL$-bundle associated with $\hyperb(E)$ (see Definition~\ref{def:SLbundle}). Then one has
	\[
	b_{2n+1}(\hyperb(E))=e(E\oplus E^\vee,\lambda).
	\]
	It follows from Theorem~\ref{thm:Euler_torsion} that $\eta \cup e(E\oplus E^\vee,\lambda)=0$. 
\end{proof}

\begin{corollary} \label{cor:Pontryagin}
	Suppose that $S=\Spec k$ for a field $k$ of characteristic different from $2$. Let $A\in \SH(k)$ be an $SL$-oriented spectrum. Then the following holds after inverting $\eta$ in the coefficients, i.e. in $A^{*,*}(X)[\eta^{-1}]$.
	\begin{enumerate}
		\item 
		$b_{2i+1}(\hyperb(E))=0, i\ge 0,$ for a vector bundle $E$ over $X\in\Sm_k$.
		\item 
		$p_t(E\oplus E')=p_t(E)p_t(E')$ for vector bundles $E,E'$ over $X\in\Sm_k$.
		\item 
		$p_i(E)=0$ for a rank $n$ vector bundle $E$ over $X\in\Sm_k$ and $i>\tfrac{n}{2}$.
	\end{enumerate}
\end{corollary}
\begin{proof}
	We work with $\eta$-inverted coefficients, i.e. in $A^{*,*}(X)[\eta^{-1}]$.
	
	(1) It follows from Lemma~\ref{lem:torsionBorel} that 
	\[
	b_t(\hyperb((\det E)^\vee))=1+b_1(\hyperb((\det E)^\vee))t=1.
	\]
	Then the multiplicativity property of the total Borel class yields
	\[
	b_t(\hyperb(E\oplus (\det E)^\vee))=b_t(\hyperb(E))b_t(\hyperb((\det E)^\vee))=b_t(\hyperb(E)).
	\]
	The equality $b_{2i+1}(\hyperb(E\oplus (\det E)^\vee))=0$ follows from \cite[Corollary~3]{An15} whence the claim.
	
	(2) It follows from the above that
	\[
	p_{-t^2}(E)=b_t(\hyperb(E)),\, p_{-t^2}(E')=b_t(\hyperb(E')),\, p_{-t^2}(E\oplus E')=b_t(\hyperb(E\oplus E')).
	\]
	The claim follows from the multiplicativity property of the total Borel class.
	
	(3) We have $p_t((\det E)^\vee)=1$. It follows from the above that
	\[
	p_t(E)=p_t(E)p_t((\det E)^\vee)=p_t(E\oplus (\det E)^\vee).
	\]
	\cite[Corollary~3]{An15} yields $p_i(E\oplus (\det E)^\vee)=0$ for $i>\tfrac{n+1}{2}$ whence $p_i(E)=0$ for $i>\tfrac{n+1}{2}$. The only remaining case is $n=2m+1$, $i=m+1$. In this case we have
	\[
	p_{m+1}(E)=p_{m+1}(E\oplus (\det E)^\vee)=e(E\oplus (\det E)^\vee,\lambda)^2
	\]
	where $\lambda\colon \det (E\oplus (\det E)^\vee)\rightiso \triv_X$ is the canonical isomorphism and the last equality is given by \cite[Corollary~3]{An15}. Theorem~\ref{thm:Euler_torsion} yields $e(E\oplus (\det E)^\vee,\lambda)=0$ whence the claim.
\end{proof}

\begin{remark}
	The assumption $S=\Spec k$ with $k$ being a field of characteristic different from $2$ arises from the same assumption of \cite{An15}. It seems that this assumption is redundant and all the results of the loc. cit. as well as the above corollary should hold over a general quasi-compact quasi-separated scheme. Moreover, almost all the reasoning given in \cite{An15} works over a general base.
\end{remark}

\end{document}